\documentclass[a4paper,10pt]{article}
\usepackage{graphicx}
\usepackage{comment}
\usepackage{amssymb}
\usepackage{amsmath}
\usepackage{nicefrac}
\usepackage{graphicx}
\usepackage{dcolumn}
\usepackage{bm}
\usepackage{comment}
\usepackage{multirow}
\usepackage{cite}
\usepackage{xcolor}
\usepackage{url}
\usepackage{mathrsfs}
\usepackage{supertabular}
\usepackage[normalem]{ulem}
\usepackage{lscape}
\usepackage{soul}
\usepackage{subcaption} 
\usepackage{bold-extra}
\usepackage{hyperref}
\usepackage{latexsym,amsthm} 

\newtheorem{theorem}{Theorem}

\theoremstyle{definition}
\newtheorem{definition}{Definition}

\begin{document}
            
\huge

\begin{center}
Identities for Permutations with Fixed Points
\end{center}

\vspace{0.5cm}

\large

\begin{center}
Jean-Christophe Pain$^{a,b,}$\footnote{jean-christophe.pain@cea.fr}
\end{center}

\normalsize

\begin{center}
\it $^a$CEA, DAM, DIF, F-91297 Arpajon, France\\
\it $^b$Universit\'e Paris-Saclay, CEA, Laboratoire Mati\`ere en Conditions Extr\^emes,\\
\it 91680 Bruy\`eres-le-Ch\^atel, France\\
\end{center}

\vspace{0.5cm}

\begin{abstract}
We present identities for permutations with fixed points. The formulas are based on successive derivations or integrations of the determinant of a particular matrix. 
\end{abstract}

\section{Introduction}
The number of permutations with fixed points can be obtained by induction \cite{Tian2021}. It is of interest for combinatorial interpretations of relations for the number of derangements \cite{Rakotondrajao2007}. For instance, the number of derangements  can be interpreted as the sum of the values of the largest fixed points of all non-derangements of length $n-1$ \cite{Deutsch2010}.  In
the same paper, Deutsch and Elizalde \cite{Deutsch2010} showed that the analogous sum for the smallest fixed points equals the number of permutations of length $n$ with at least two fixed points. In this short article, we derive identities for the fixed-point statistics over the symmetric group.
\begin{definition}
Let us consider, for $x\in\mathbb{R}$ and any natural number $n\geq 2$, the $n\times n$ matrix 
\begin{equation*}
  M_x =
\left[ \begin{matrix}
     x & 1 & \cdots & 1 & 1\\ 
     1 & x & \cdots & 1 & 1\\ 
\vdots & \vdots & \ddots & \vdots & \vdots\\ 
     1 & 1 & \cdots & x & 1\\
     1 & 1 & \cdots & 1 & x\\
\end{matrix} \right] .
\end{equation*}
\end{definition}

The characteristic polynomial of the matrix $-M_0$ is 
\begin{equation*}
    \chi_{-M_0}(X)=\mathrm{det}(XI_n+M_0). 
\end{equation*}
Hence, since $\chi_{-M_0}(1)=0$, it follows that $1$ is an eigenvalue, the corresponding eigenspace being
\begin{equation*}
    E_1=\mathrm{ker} (-M_0-I_n) =\mathrm{ker} (M_0+I_n).  
\end{equation*}
We have 
\begin{equation*}
    x=(x_1,\cdots, x_n)\in E_1 \;\;\;\; \Longleftrightarrow \;\;\;\; x_1+x_2+\cdots+x_n=0,
\end{equation*}
which is a hyperplane of dimension $n-1$. This implies that the multiplicity of 1 is $n-1$. In addition, the sum of all eigenvalues is equal to the trace of the matrix $-M_0$, i.e., 0, and therefore $1-n$ is also an eigenvalue. We have
\begin{align*}
    x=(x_1,\cdots, x_n)\in E_{1-n} \;\;\;\; &\Longleftrightarrow \;\;\;\; \forall i\in[1,n], \;\;\;\; (n-1)x_i=\sum_{k=1,k\ne i}^nx_k\\
    &\Longleftrightarrow \;\;\;\; x_1=x_2=\cdots=x_n\\
    &\Longleftrightarrow \;\;\;\; x=\mathrm{span}\{(1,1,\cdots, 1)\},
\end{align*} 
and the multiplicity of $1-n$ is 1. Thus
\begin{align}\label{alter}
    \mathrm{det}\ M_x&=\mathrm{det}(xI_n+M_0) \nonumber\\
    &=(x-1+n)(x-1)^{n-1}.
\end{align}
By the Leibniz formula \cite{Strang2016}, the determinant of the matrix $M_x$ gives
\begin{equation*}
    \mathrm{det}\ M_x =\sum_{\sigma\in\mathfrak{S}_n}\epsilon(\sigma)\prod_{i=1}^n(m(x))_{\sigma(i),i},
\end{equation*}
where $\epsilon(\sigma)$ is the signature of the permutation $\sigma\in\mathfrak{S}_n$ and
\begin{displaymath}
    (m(x))_{\sigma(i),i}= 
    \begin{cases}
    1, & \text{if $\sigma(i)\ne i$;}\\
    x, & \text{if $\sigma(i)=i$,}
    \end{cases}
\end{displaymath}
yielding
\begin{align}\label{eq2}
   \mathrm{det}\ M_x &=\sum_{\sigma\in\mathfrak{S}_n}\epsilon(\sigma)\prod_{i\in\mathrm{fix}(\sigma)}x\nonumber\\
   &=\sum_{\sigma\in\mathfrak{S}_n}\epsilon(\sigma)x^{\mathrm{fix}(\sigma)},
\end{align}
where $\mathrm{fix}(\sigma)$ is the number of fixed points of $\sigma$.

\section{First relation by successive derivatives of the determinant}

\begin{theorem}
Let $x\in\mathbb{R}$ and $n$ be a natural number. We have
\begin{equation}\label{theo2}
    \sum_{\sigma\in\mathfrak{S}_n}\epsilon(\sigma)\mathrm{fix}(\sigma)\left(\mathrm{fix}(\sigma)-1\right)\cdots\left(\mathrm{fix}(\sigma)-k+1\right)x^{\mathrm{fix}(\sigma)-k}=(x-1)^{n-k-1}\frac{n!(x+n-k-1)}{(n-k)!}.
\end{equation}
\end{theorem}

\begin{proof}
By differentiating Eq.\ (\ref{eq2}) $k$ times, for $k\geq 1$, one gets
\begin{equation}\label{eqbis}
    \left(\mathrm{det}\ M_x\right)^{(k)}=\sum_{\sigma\in\mathfrak{S}_n}\epsilon(\sigma)\mathrm{fix}(\sigma)\left(\mathrm{fix}(\sigma)-1\right)\cdots\left(\mathrm{fix}(\sigma)-k+1\right)x^{\mathrm{fix}(\sigma)-k}.
\end{equation}
The Leibniz formula for the multiple derivative of a product gives, using Eq.\ (\ref{alter}):

\begin{equation*}
    \left(\mathrm{det}\ M_x \right)^{(k)}=\sum_{p=0}^k\binom{k}{p}(x-1+n)^{(p)}\left((x-1)^{n-1}\right)^{(k-p)}
\end{equation*}
and
\begin{displaymath}
    (x+1-n)^{(p)}=
\begin{cases}
x-1+n, & \text{if $p=0$;}\\
1, & \text{if $p=1$;}\\
0, & \text{otherwise},\\
\end{cases}
\end{displaymath}
and thus
\begin{align}\label{detmx}
    \left(\mathrm{det}\ M_x\right)^{(k)}&=(x-1+n)\left((x-1)^{n-1}\right)^{(k)}+k\left((x-1)^{n-1}\right)^{(k-1)}\nonumber\\
    &=(x-1)^{n-k-1}\frac{(n-1)!}{(n-k-1)!}\left(x-1+n+\frac{k(x-1)}{n-k}\right)\nonumber\\
    &=(x-1)^{n-k-1}\frac{n!(x+n-k-1)}{(n-k)!},
\end{align}
which, combined with Eq.\ (\ref{eqbis}), completes the proof.
\end{proof}

For $x=1$, according to Eq.\ (\ref{detmx}), one has
\begin{equation*}
    \left(\mathrm{det}\ M_1\right)^{(k)}=n!\,\delta_{k,n-1},
\end{equation*}
where $\delta_{i,j}$ is the usual Kronecker delta symbol. Setting $k=n-1$ in Eq.\ (\ref{eqbis}) yields
\begin{equation*}
    \sum_{\sigma\in\mathfrak{S}_n}\epsilon(\sigma)\mathrm{fix}(\sigma)\left(\mathrm{fix}(\sigma)-1\right)\left(\mathrm{fix}(\sigma)-2\right)\cdots\left(\mathrm{fix}(\sigma)-n+2\right)=n!.
\end{equation*}
As another example, setting $x=2$ in Eq.\ (\ref{theo2}) gives 
\begin{equation*}
    \sum_{\sigma\in\mathfrak{S}_n}\epsilon(\sigma)\mathrm{fix}(\sigma)\left(\mathrm{fix}(\sigma)-1\right)\left(\mathrm{fix}(\sigma)-2\right)\cdots\left(\mathrm{fix}(\sigma)-k+1\right)2^{\mathrm{fix}(\sigma)-k}=\frac{n!(n-k+1)}{(n-k)!}.
\end{equation*}

\section{Second identity by successive integrations of the determinant}

\begin{theorem}
For $k\in\mathbb{N}^*$ we have
\begin{equation}\label{eq8}
    \sum_{\sigma\in\mathfrak{S}_n}\epsilon(\sigma)\frac{(\mathrm{fix}(\sigma))!}{(\mathrm{fix}(\sigma)+k)!}=(-1)^{n+1}\frac{(n^2+(k-1)n-(k-1))}{(k-1)!(n+k-1)(n+k)}.
\end{equation}
\end{theorem}

\begin{proof}
The left-hand side of Eq.\ (\ref{eq8})
is obtained by successive integrations of 
\begin{equation*}
    \sum_{\sigma\in\mathfrak{S}_n}\epsilon(\sigma)x^{\mathrm{fix}(\sigma)}.
\end{equation*}
According to Eq.\ (\ref{eq2}), it is therefore equal to $P_{n,k}(1)$ with, for $k\geq 2$:
\begin{equation}\label{Pnkx}
    P_{n,k}(x)=\int_0^x\left(\int_0^{x_{k-1}}\left(\int_0^{x_{k-2}}\cdots\left(\int_0^{x_2}\left(\int_0^{x_1}\mathrm{det}(M_{u})\mathrm{d}u\right)\mathrm{d}x_1\right)\cdots\right)\mathrm{d}x_{k-2}\right)\mathrm{d}x_{k-1},
\end{equation}
or also
\begin{equation*}
    P_{n,k}(x)=\int_0^xP_{n,k-1}(u)\mathrm{d}u.
\end{equation*}
Since $\mathrm{det}(M_x)=(x-1)^n+n(x-1)^{n-1}$, one has in particular
\begin{equation}\label{p1x}
    P_{n,1}(x)=\int_0^x\mathrm{det}(M_{u})\mathrm{d}u=(-1)^{n+1}\frac{n}{n+1}+\frac{(x-1)^{n+1}}{n+1}+(x-1)^n,
\end{equation}
as well as
\begin{equation*}
    P_{n,2}(x)=\int_0^x\left(\int_0^{x_1}\mathrm{det}(M_{u})\mathrm{d}u\right)\mathrm{d}x_1=\int_0^xP_{n,1}(u)\mathrm{d}u
\end{equation*}
yielding
\begin{equation}\label{p2x}
    P_{n,2}(x)=(-1)^{n+1}\frac{nx}{(n+1)}+\frac{(-1)^{n+1}}{(n+1)(n+2)}+\frac{(-1)^{n}}{(n+1)}+\frac{(x-1)^{n+2}}{(n+1)(n+2)}+\frac{(x-1)^{n+1}}{(n+1)}.
\end{equation}  
Applying the same procedure, the polynomial for $k=3$ gives
\begin{equation*}
    P_{n,3}(x)=\int_0^x\left(\int_0^{x_1}\left(\int_0^{x_2}\mathrm{det}(M_{u})\mathrm{d}u\right)\mathrm{d}x_1\right)\mathrm{d}x_2=\int_0^xP_{n,2}(u)\mathrm{d}u
\end{equation*}
leading to
\begin{align}\label{p3x}
    P_{n,3}(x)&=(-1)^{n+1}\frac{nx^2}{2(n+1)}\nonumber\\
    &+(-1)^{n+1}\frac{x}{(n+1)(n+2)}+(-1)^{n}\frac{x}{(n+1)}\nonumber\\
    &+\frac{(-1)^{n}}{(n+1)(n+2)(n+3)}+\frac{(-1)^{n+1}}{(n+1)(n+2)}\nonumber\\
    &+\frac{(x-1)^{n+3}}{(n+1)(n+2)(n+3)}+\frac{(x-1)^{n+2}}{(n+1)(n+3)}.
\end{align}  
Setting $k=1$, 2 and 3 in Eqs.\ (\ref{p1x}), (\ref{p2x}) and (\ref{p3x}) gives (all the terms proportional to powers of $x-1$ vanish):
\begin{equation*}
    P_{n,1}(1)=(-1)^{n+1}\frac{n}{(n+1)},
\end{equation*}
\begin{equation*}
    P_{n,2}(1)=(-1)^{n+1}\frac{n}{(n+1)}+\frac{(-1)^{n+1}}{(n+1)(n+2)}+\frac{(-1)^{n}}{(n+1)},
\end{equation*}  
and
\begin{align*}
    P_{n,3}(1)&=(-1)^{n+1}\frac{n}{2(n+1)}\nonumber\\
    &+\frac{(-1)^{n+1}}{(n+1)(n+2)}+\frac{(-1)^{n}}{(n+1)}\nonumber\\
    &+\frac{(-1)^{n}}{(n+1)(n+2)(n+3)}+\frac{(-1)^{n+1}}{(n+1)(n+2)}.
\end{align*}  
Thus, for any value of $k\geq 2$, we have 
\begin{equation*}
    P_{n,k}(1)=\frac{(-1)^{n+1}n}{(n+1)(k-1)!}-\sum_{j=2}^k\frac{(-1)^{n+j}}{(k-j)!}\left( \frac{1}{\prod_{l=n+1}^{n+j}l}-\frac{1}{\prod_{l=n+1}^{n+j-1}l}\right),
\end{equation*}
with
\begin{equation*}
    \sum_{j=2}^k\frac{1}{(k-j)!}\frac{(-1)^{n+j}}{\prod_{l=n+1}^{n+j}l}=\sum_{j=2}^k\frac{(-1)^{n+j}}{(k-j)!}\frac{n!}{(n+j)!}.
\end{equation*}
Let us set
\begin{align*}
    S_n(k)=\sum_{j=2}^k\frac{(-1)^{n+j}}{(k-j)!(n+j)!},
\end{align*}
so that we can write
\begin{equation}\label{eq10}
    P_{n,k}(1)=\frac{(-1)^{n+1}n}{(n+1)(k-1)!}-n!\,S_n(k)-n!\,S_{n-1}(k).
\end{equation}
Let us now assume that $k\geq 2$. We have
\begin{align*}
    S_n(k)&=\sum_{j=2}^k\frac{(-1)^{n+j}}{(k-j)!(n+j)!}\\
    &=\frac{1}{(n+k)!}\sum_{j=2}^k\frac{(-1)^{n+j}(n+k)!}{(k-j)!(n+j)!}.
\end{align*}
The latter quantity is equal to
\begin{align*}
    S_n(k)&=\frac{1}{(n+k)!}\sum_{j=2}^k\binom{n+k}{n+j}(-1)^{n+j}\\
    &=\frac{1}{(n+k)!}\sum_{m=n+2}^{n+k}\binom{n+k}{m}(-1)^{m}\\
    &=\frac{1}{(n+k)!}\left(\sum_{m=0}^{n+k}(-1)^m\binom{n+k}{m}-\sum_{m=0}^{n+1}(-1)^m\binom{n+k}{m}\right)\\
    &=-\frac{1}{(n+k)!}\sum_{m=0}^{n+1}(-1)^m\binom{n+k}{m}\\
    &=-\frac{T_n(k)}{(n+k)!},
\end{align*}
with
\begin{equation*}
    T_n(k)=\sum_{m=0}^{n+1}\binom{n+k}{m}(-1)^m.
\end{equation*}
Since
\begin{equation*}
    \binom{n+k}{m}=\binom{n+k-1}{m}+\binom{n+k-1}{m-1},
\end{equation*}
we get
\begin{align*}
    T_n(k)&=\sum_{m=1}^{n+1}(-1)^m\binom{n+k-1}{m}+\sum_{m=1}^{n+1}(-1)^m\binom{n+k-1}{m-1}+1\\
    &=\sum_{m=1}^{n+1}(-1)^m\binom{n+k-1}{m}+\sum_{m=0}^{n}(-1)^{m+1}\binom{n+k-1}{m}+1\\
    &=(-1)^{n+1}\binom{n+k-1}{n+1}-1+1\\
    &=(-1)^{n-1}\binom{n+k-1}{n+1}.
\end{align*}
We therefore obtain 
\begin{align*}
    S_n(k)&=\frac{(-1)^{n}(n+k-1)!}{(n+1)!(k-2)!(n+k)!}\\
    &=\frac{(-1)^{n}}{(n+1)!(k-2)!(n+k)},
\end{align*}
which, combined with Eq.\ (\ref{eq10}), yields
\begin{align*}
    P_{n,k}(1)&=\frac{(-1)^{n+1}n}{(n+1)(k-1)!}-n!\frac{(-1)^{n}}{(n+1)!(k-2)!(n+k)}-n!\frac{(-1)^{n+1}}{n!(k-2)!(n-1+k)}\\
    &=(-1)^{n+1}\frac{(n^2+(k-1)n-(k-1))}{(k-1)!(n+k-1)(n+k)},
\end{align*}
which completes the proof.
\end{proof}

For $k=1$, Eq.\ (\ref{eq8}) becomes
\begin{equation*}
    \sum_{\sigma\in\mathfrak{S}_n}\frac{\epsilon(\sigma)}{(\mathrm{fix}(\sigma)+1)}=(-1)^{n+1}\frac{n}{(n+1)},
\end{equation*}
while for $k=2$ one has
\begin{align*}
    \sum_{\sigma\in\mathfrak{S}_n}\frac{\epsilon(\sigma)}{(\mathrm{fix}(\sigma)+1)(\mathrm{fix}(\sigma)+2)}&=(-1)^{n+1}\frac{(n+\phi)(n+1-\phi)}{(n+1)(n+2)}\\
    &=(-1)^{n+1}\frac{(n+\phi)(n-1/\phi)}{(n+1)(n+2)},
\end{align*}
where $\phi$ is the golden ratio, and in the $k=3$ case:
\begin{equation*}
    \sum_{\sigma\in\mathfrak{S}_n}\frac{\epsilon(\sigma)}{(\mathrm{fix}(\sigma)+1)(\mathrm{fix}(\sigma)+2)(\mathrm{fix}(\sigma)+3)}=(-1)^{n+1}\frac{\left((n+1)^2-3\right)}{2(n+2)(n+3)}.
\end{equation*}
Here also, as in the first section, we have set $x=1$, but of course, further sum rules can be obtained by taking different values of $x$ in Eq.\ (\ref{Pnkx}) and using
\begin{equation*}
    P_{n,k}(x)=\sum_{\sigma\in\mathfrak{S}_n}\epsilon(\sigma)\frac{(\mathrm{fix}(\sigma))!}{(\mathrm{fix}(\sigma)+k)!}x^{\mathrm{fix}(\sigma)+k}.
\end{equation*}

\section{Conclusion}

We introduced two families of identities concerning permutations with fixed points, obtained respectively through successive differentiation and repeated integration of the determinant of a specific matrix. Beyond these results, the methods presented here may lead to the derivation of further sum rules.

\bigskip
\hrule
\bigskip

\noindent 2020 {\it Mathematics Subject Classification}:
05A05 ; 05A15.

\noindent \emph{Keywords: } permutation, fixed point, sum rule, linear algebra, determinant.

\bigskip
\hrule
\bigskip

\noindent \emph{Concerned with OEIS sequences: } A000166 and A008290.

\end{document}